\documentclass[10pt]{amsart}
\usepackage[margin=1.2in,marginparsep=0.1in,marginparwidth=1in]{geometry}
\usepackage{amssymb,amsmath,amsthm,amstext,amscd,latexsym,graphics,graphicx,bbm,caption}
\usepackage[usenames,dvipsnames,svgnames,table]{xcolor}
\usepackage[plainpages=false,colorlinks=true, pagebackref]{hyperref}
\usepackage{tikz}
\usetikzlibrary{positioning}

\tikzset{>=stealth}
\hypersetup{citecolor=Sepia,linkcolor=blue, urlcolor=blue}

\usepackage{cite}   %upright parenthesis for cite

\makeatletter
\def\@tocline#1#2#3#4#5#6#7{\relax
  \ifnum #1>\c@tocdepth % then omit
  \else
    \par \addpenalty\@secpenalty\addvspace{#2}%
    \begingroup \hyphenpenalty\@M
    \@ifempty{#4}{%
      \@tempdima\csname r@tocindent\number#1\endcsname\relax
    }{%
      \@tempdima#4\relax
    }%
    \parindent\z@ \leftskip#3\relax \advance\leftskip\@tempdima\relax
    \rightskip\@pnumwidth plus4em \parfillskip-\@pnumwidth
    #5\leavevmode\hskip-\@tempdima
      \ifcase #1
       \or\or \hskip 2em \or \hskip 2em \else \hskip 3em \fi%
      #6\nobreak\relax
    \dotfill\hbox to\@pnumwidth{\@tocpagenum{#7}}\par
    \nobreak
    \endgroup
  \fi}
\makeatother

\newtheorem{intro-thm}{Theorem}[]
\theoremstyle{plain}
\newtheorem{thm}{Theorem}[section]
\newtheorem{theorem}[thm]{Theorem}

\newtheorem{lemma}[thm]{Lemma}

\newtheorem{proposition}[thm]{Proposition}

\theoremstyle{definition}
\newtheorem{remark}[thm]{Remark}

\newtheorem{definition}[thm]{Definition}

\newcommand{\Ker}{{\rm Ker  }}

\renewcommand{\tilde}{\widetilde}
\newcommand{\sC}{{\mathcal C}}

\newcommand{\sF}{{\mathcal F}}
\newcommand{\sG}{{\mathcal G}}

\newcommand{\sS}{{\mathcal S}}

\newcommand{\sX}{{\mathcal X}}
\newcommand{\sY}{{\mathcal Y}}

\newcommand{\A}{{\mathbb A}}

\newcommand{\G}{{\mathbb G}}

\newcommand{\N}{{\mathbb N}}
\renewcommand{\P}{{\mathbb P}}

\newcommand{\Sh}{\sS h}

\input{xy}
\xyoption{all}

\begin{document}

\title[Hurewicz map]{The Hurewicz map in motivic homotopy theory}


\author{Utsav Choudhury}

\author{Amit Hogadi}

\begin{abstract} For an $\A^1$-connected pointed simplicial sheaf $\sX$ over a perfect field $k$,
we prove that the Hurewicz map $\pi_1^{\A^1}(\sX) \to H_1^{\A^1}(\sX)$ is surjective. We also observe that the Hurewicz map for $\P^1_k$ is the abelianisation map. In the course of proving this result, we also show that for any morphism $\phi$ of strongly $\A^1$-invariant sheaves of groups, the image and kernel of $\phi$ are also strongly $\A^1$-invariant.
\end{abstract}

\maketitle

\section{Introduction}

For a field $k$, let $Sm/k$ denote the category of smooth $k$-varieties with Nisnevich topology. Let $\Delta^{op}Sh(Sm/k)$ denote the category of simplicial sheaves on the category $Sm/k$. This cateogry with its $\A^1$-model structure as defined in \cite{morel-voevodsky} is one of the main objects of study in $\A^1$-homotopy theory. For any pointed simplicial sheaf $\sX$ in $\Delta^{op}Sh(Sm/k)$ one defines the $\A^1$-homotopy group sheaves, $\pi_i^{\A^1}(\sX)$, to be the sheaves of simplicial homotopy groups of a fibrant replacement of $\sX$ in the $\A^1$-model structure. Morel, in his foundational work in \cite[Ch. 6]{morel} has defined, for every integer $i$, $\A^1$-homology groups $H_i^{\A^1}(\sX)$ and canonical Hurewicz morphisms $$\pi_i^{\A^1}(\sX) \to H_i^{\A^1}(\sX)$$
The above maps are analogous to the Hurewicz map that we have in topology. In the topological setup, the Hurewicz morphism for $i=1$ is known to be the abelianisation when the underlying space is connected. We will refer to this result as the Hurewicz theorem. Hurewicz theorem is expected in $\A^1$-homotopy theory (see \cite[6.36]{morel}), but not yet known. However the following theorem by Morel is the closest known result to the Hurewicz theorem.

\begin{theorem}\cite[6.35]{morel}\label{morel-hurewicz}
For a connected simplicial sheaf $\sX$, the Hurewicz morphism $$ \pi_1^{\A^1}(\sX) \to H_1^{\A^1}(\sX)$$ is a universal map to a strictly $\A^1$-invariant sheaf of abelian groups. 
\end{theorem}

Recall that a sheaf of groups $\sG$ is called strongly $\A^1$-invariant if for $i=0,1$ the maps
$$ H^i(U, \sG) \to H^i(U\times \A^1, \sG)$$ are bijective for all $U$ in $Sm/k$. If $\sG$ is abelian, then it is called strictly $\A^1$-invariant if 
the above isomorphism holds for all $i\geq 0$. $\pi_1^{\A^1}(\sX)$ is strongly $\A^1$-invariant and $H_1^{\A^1}(\sX)$ is known to be strictly $\A^1$-invariant (see \cite[6.1, 6.23]{morel}). 

In topology, the surjectivity of the Hurewicz map is almost a direct consequence of the definitions. This is not the case in $\A^1$-homotopy theory. The main source of difficulty lies in the non-explicit nature of $\A^1$-fibrant replacements; non-explicit from the viewpoint of making explicit calculations.
In this paper we prove this surjectivity by using Giraud's theory of non-abelian cohomology. 
\begin{theorem} \label{hurewicz}
Let $k$ be a perfect field and $\sX$ be a pointed simplicial sheaf on $Sm/k$ in the Nisnevich topology. Then the Hurewicz map 
$ \pi_1^{\A^1}(\sX) \to H_1^{\A^1}(\sX) $ is surjective.
\end{theorem}

The above theorem will be deduced from the following result, which is of independent interest.

\begin{thm} \label{main theorem} Let $k$ be a perfect field. Let $G$ be a strongly $\A^1$-invariant sheaf of groups on $Sm/k$ and $G\to H$ be an epimorphism. Then $H$ is strongly $\A^1$-invariant iff it is $\A^1$-invariant. 
\end{thm}

\begin{remark} If $k$ is perfect field, a theorem of Morel \cite[5.46]{morel} says that any strongly $\A^1$-invariant Nisnevich sheaf of abelian groups on $Sm/k$ is also strictly $\A^1$-invariant. Unfortunately it is not yet known if this statement holds for imperfect fields. This is the sole reason for assuming $k$ to be perfect in Theorems \ref{main theorem} and \ref{hurewicz}. Also note that strongly $\A^1$-invariant is a stronger notion than just $\A^1$-invariant. In particular, there exists $\A^1$-invariant sheaves which are not strongly $\A^1$-invariant (see \cite[Lemma 5.6]{ch}).
\end{remark}

For a morphism of strongly $\A^1$-invariant abelian sheaves over a perfect field, the kernel and image of the morphism are also strongly $\A^1$-invariant. This result is a consequence of a nontrivial theorem of Morel  (see \cite[6.24]{morel}) that the category of strongly $\A^1$-invariant sheaves of abelian groups is an abelian category, as it is obtained as a heart of a $t$-structure. The theorem below, can be viewed as a generalization of this result for non-abelian strongly $\A^1$-invariant sheaves.  Moreover the proof of this generalization is completely different and is more direct in the sense that it does not appeal to the existence of $t$-structures.

\begin{theorem}
Let $G\xrightarrow{\phi} H$ be a morphism of strongly $\A^1$-invariant sheaves of groups. Then the image and the kernel of $\phi$ are strongly $\A^1$-invariant. 
\end{theorem}
\begin{proof}
The image ${\rm Image}(\phi)$ is $\A^1$-invariant, since it is a subsheaf of an $\A^1$-invariant sheaf $H$. Thus by Theorem \ref{main theorem} it is strongly $\A^1$-invariant. The kernel $K$ is strongly $\A^1$-invariant as it fits in the following exact sequence 
$$ 1\to K \to G \to {\rm Image(\phi)} \to 1$$ 
where the other two sheaves are strongly $\A^1$-invariant. 
\end{proof}

\noindent {\bf Acknowledgement}: We thank Fabien Morel, Tom Bachmann and O. R{\"o}ndigs for their comments.

\section{Preliminaries on Gerbes and Giraud's non-abelian cohomology}

Let $\sC$ be any small site (e.g. $Sm/k$ with Nisnevich topology) and $\Delta^{op}(\sC)$ be the category of simplicial sheaves on $\sC$. The goal of this section is to recall the main results of Giraud on non-abelian cohomology. Everything in this section is a subset of \cite{giraud}.  We work with the following definition of a gerbe.

\begin{definition} A simplicial sheaf $\sX$ on $\sC$ is called a gerbe if it is connected and if for any $U\in \sC$ and any $x\in \sX(U)$, the homotopy sheaves of groups $\pi_i(\sX_{|U}, x) = 0$ for all $i \in \N$.
\end{definition}

Given any simplicial sheaf $\sX$ (not necessarily a gerbe), one gets a category fibered in groupoids over $\sC$ defined by the fundamental groupoid construction: i.e. for every $U\in \sC$, the fiber category $\sX_{|U}$ is the fundamental groupoid of the space $\sX(U)$,  a category whose objects are elements of $\sX(U)$ and morphisms are paths up to homotopy. 
This category fibered in groupoids is in fact a gerbe in the sense of \cite[3.2]{laumon} if $\sX$ is connected. If $\sX$ was a gerbe to start with, then it can be recovered, up to weak equivalence, using this category fibered in groupoids using the simplicial nerve construction. 
A gerbe $\sX$, is called neutral if it has a global section. In this case, by making a choice of a global section, one can define the fundamental group of $\sX$. Since $\sX$ is connected, a different choice gives a fundamental group which can be canonically identified with the previous one, modulo an inner automorphism.

This motivates the following definition by Giraud.
\begin{definition}\cite[1.1.3]{giraud}
For two sheaves of groups $F$ and $G$, let ${\rm Isex}(F,G)$ denote the set of isomorphisms from $F$ onto $G$ modulo the action of inner automorphisms of $F$ (acting on the left) and the action of inner automorphisms of $G$. 
\end{definition}

Consider the pre-stack whose objects over $U$ are sheaves of groups over $U$ (small w.r.t to a fixed universe) with morphisms between $F$ and $G$  defined as elements of ${\rm Isex}(F,G)$. One can stackify this pre-stack (see \cite{laumon}) and objects of this stack are called bands. In particular, every sheaf of groups defines a band. Since every band is represented locally by a sheaf of groups, all those concepts related to sheaves of groups which are local in nature (e.g. exact sequence, epimorphism, kernel, center) also make sense for bands. It is a simple exercise to show that the center of a band is necessarily represented by a sheaf of groups. The 'fundamental group' of any gerbe $\sX$ (neutral or not) is always defined  as a band. For a band $L$, a gerbe banded by $L$ (or simply an $L$-gerbe) will mean a gerbe together with an isomorphism of $L$ with the band defined by $\sX$. An equivalence of $L$-gerbes means an equivalence of the gerbes compatible with the given isomorphisms of their bands with $L$. 

\begin{definition}\cite[3.1.1]{giraud}
For a band $L$ on a site $\sC$, let $H^2(\sC, L)$ or simply $H^2(L)$ denote the equivalence class of $L$-gerbes. The subset represented by neutral classes in $H^2(\sC, L)$ is denoted by $H^2(\sC, L)'$ or simply ($H^2(L)'$).  
\end{definition}

\begin{remark}
Note that $H^2(L)'$ is non-empty if and only if $L$ can be represented by a sheaf of groups (see \cite[3.2.4]{giraud}), in which case it is a singleton set, as can be seen for e.g. by Theorem \ref{all-abelian} stated below.
\end{remark}

\begin{remark}
If a band $L$ is representable by a sheaf of abelian groups $A$, then $H^2(L)$ defined above is in canonical bijection with the $H^2(A)$ as defined by sheaf cohomology \cite[3.4]{giraud}.
\end{remark}

Given an exact sequence of sheaves of groups 
$$ 1 \to A \xrightarrow{a} B \xrightarrow{b} C \to 1$$
one has a long exact sequence 
$$ 1 \to H^0(A) \to H^0(B) \to H^0(C) \to H^1(A) \to H^1(B) \to H^1(C).$$ 
One of the goals of introducing the non-abelian $H^2$ is to extend this exact sequence on the right. We first note that $B$ acts on itself by inner automorphisms. Since $A$ is normal in $B$, this action also induces an action on $A$. Thus $B$ also acts on ${\mathsf {band}}(A)$, where ${\mathsf {band}}(A)$ denotes the band defined by $A$. This action factors through $C$.  

We need the following definition to state the result extending the above long exact sequence to $H^2$:

\begin{definition}\cite[4.2.3]{giraud} \label{o}
For an  epimorphism of bands $v: M \to N$, one defines a pointed set 
$${\mathsf O}(v) := N(v)/R$$ where  
\begin{enumerate}
\item $N(v)$ is the set of all triples $(K,L,u)$ where $K$ is a gerbe with band $L$ and $u:L\to M$ is a monomorphism which makes the following sequence exact 
$$ 1 \to L \xrightarrow{u} M \xrightarrow{v} N \to 1.$$
\item $R$ is the equivalence relation defined by declaring $(K,L,u) \sim (K',L',u')$ if there exists a morphism of gerbes $\alpha: K\to K'$ such that the induced morphism $\alpha: L\to L'$ on bands makes the following diagram commute 
$$\xymatrix{
 L \ar[d]_\alpha \ar[r]^u & M \ar@{=}[d] \\
 L\ar[r]^{u'} & M
}$$
\item ${\mathsf O}(v)'$ denotes the subset of ${\mathsf O}(v)$ defined by all those $(K,L,u)$ where $K$ is a neutral $L$-gerbe. 
\end{enumerate} 

\end{definition}

Now one has the following general result.

\begin{theorem}\cite[4.2.8, 4.2.10]{giraud} 
Given an exact sequence of sheaves of groups $ 1 \to A \xrightarrow{f} B \xrightarrow{g} C \to 1$, we have the following long exact sequence 
$$1 \to H^0(A) \to H^0(B) \to H^0(C) \to H^1(A) \to H^1(B) \to H^1(C)\xrightarrow{d} {\mathsf O}(g) \to H^2(B) \to H^2(C) .$$
where the map $d$ is as defined in  \cite[4.7.2.4]{giraud} and exactness of the sequence is defined similar to that in the case of pointed sets with subsets ${\mathsf O}(g)'$ and $H^2(B)'$, $H^2(C)'$ playing the role of base points.  When the action of $C$ on ${\mathsf {band}}(A)$ is trivial, on has a canonical bijection ${\mathsf O}(g) \cong H^2(A)$.
\end{theorem}

Another important result of Giraud we need is the following.
\begin{theorem}\cite[3.3.3]{giraud} \label{all-abelian}
Let $L$ be a band and $C$ be its center. Then one has a canonical action of $H^2(C)$ on $H^2(L)$ which is free and transitive. 
\end{theorem}
This theorem, loosely speaking, says that the non-abelian cohomology set $H^2(L)$ is essentially all abelian as it "comes from" its center. Note however that if $H^2(L)$ has no class represented by a neutral gerbe, then there is no canonical bijection between $H^2(L)$ and $H^2(C)$. \\

\noindent The following is a direct consequence of the above theorem and the definition of 
\begin{lemma} \label{o-trivial}
Let $1 \to A\xrightarrow{f} B \xrightarrow{g} C \to 1$ be an exact sequence of sheaves of groups. Assume that $H^2(A)$ is trivial. Then ${\mathsf O}(g)' = {\mathsf O}(g)$.
\end{lemma}
\begin{proof}
Let $(K, L, u)$ be any triple where $L$ is a band which fits in the exact sequence 
$$ 1 \to L \xrightarrow{u} {\mathsf {band}}(B) \to {\mathsf {band}}(C) \to 1$$
and $K$ is an $L$-gerbe. To prove the theorem it is enough to show that $K$ is neutral. However we note that center $Z(L)$ coincides with $Z(A)$,  the center of $A$. The result then directly follows from the above theorem. 
\end{proof}

\section{Strong $\A^1$-invariance of the center and applications}
The goal of this section is to prove the theorems mentioned in the introduction. We start by proving the following. 

\begin{thm} \label{center} Let $G$ be a strongly $\A^1$-invariant sheaf of groups on $Sm/k$. Then $Z(G)$, the center of $G$ is also strongly $\A^1$-invariant.
\end{thm}
\begin{proof}
Since $G$ is strongly $\A^1$-invariant, $BG$ is $\A^1$-local. By choosing a simplicially fibrant model for $BG$ we may further assume, without loss of generality, that $BG$ is $\A^1$-fibrant. To show $Z(G)$ is strongly $\A^1$-invariant, we need to show $BZ(G)$ is $\A^1$-local. Let $BZ(G) \xrightarrow{u} \sX$ be an $\A^1$-fibrant replacement (in the category of pointed spaces). Thus $u$ is a trivial $\A^1$-cofibration. To prove the theorem, it suffices to show that $u$ is a simplicial weak equivalence. Equivalently, it suffices to show that the map on sheaves of fundamental groups 
$$ \pi_1(BZ(G)) \longrightarrow \pi_1(\sX) (=: H)$$ is an isomorphism.  Since $BG$ is $\A^1$-fibrant and $u$ is a trivial $\A^1$-cofibration, we have a factorization $h$ as below 
$$\xymatrix{
BZ(G) \ar[d]_-u  \ar[r]& BG \\
\sX \ar[ru]_-{\exists h} & 
}$$ 
which gives a commutative diagram of the maps induces on the fundamental groups  

$$\xymatrix{
Z(G) \ar[d]_-{u_*} \ar[r] & G \\
H \ar[ru]_-{h_*} & 
}$$
\end{proof}
In the above diagram, if we show that the image of $h_*$ is $Z(G)$, then it will follow that $Z(G)$ is a retract of the strongly $\A^1$-invariant sheaf $H$ and hence is itself strongly $\A^1$-invariant. Thus it suffices to show that image of $h_*$ is contained in $Z(G)$. 
This is equivalent to showing that for every smooth $k$-scheme $U$ and an element $g\in G(U)$, the map $h_{|U}$ is homotopic to the composite of  
$$ \sX_{|U} \xrightarrow{h_{|U}} BG_{|U} \xrightarrow{x\mapsto gxg^{-1}} BG_{|U}.$$
\noindent But note that the base change functor from $\Delta^{op}(\Sh(k)) \to  \Delta^{op}(\Sh(U))$ preserves trivial $\A^1$-cofibrations since it is a left quillen functor. Moreover it also preserves $\A^1$-fibrations.   Thus 
$$BZ(G)_{|U} \xrightarrow{\nu_{|U}} \sX_{|U} $$ is a trivial $\A^1$-cofibration. Since $Z(G)$ is the center of $G$, the map $BZ(G)_{|U} \xrightarrow{\nu_{|U}} BG_{|U}$ is homotopic (in fact equal) to the composite 
$$ BZ(G)_{|U} \to BG_{|U} \xrightarrow{x\mapsto gxg^{-1}} BG_{|U}.$$
The proof now follows from commutative diagram below, using the fact that $\nu_{\scriptscriptstyle{ |U}}$ is an $\A^1$-weak equivalence and the fact that $BG_{|U}$ is $\A^1$-local by Lemma \ref{projective-trick}.
$$\xymatrix{
BZ(G)_{|U} \ar[d]_-{\nu_{\scriptscriptstyle{ |U}}} \ar[r] & BG_{|U} \ar[r]^{x\mapsto gxg^{-1}} & BG_{|U}\\
\sX_{|U} \ar[ru]_-{h_{|U}} & & 
}$$

\begin{lemma}\label{projective-trick}
$U\in Sm/k$. Then the restriction functor $\Delta^{op}(\Sh(k)) \to  \Delta^{op}(\Sh(U))$ takes $\A^1$-fibrant objects to $\A^1$-local objects. 
\end{lemma}
\begin{proof}
Let $\sY\in \Delta^{op}(\Sh(k))$ be an $\A^1$-fibrant object. Then $\sY_{|U}$ has BG property, since BG property is defined in terms of Nisnevich distinguished triangles and every Nisnevich distinguished triangle in the category $Sm/U$ is also a Nisnevich disntinguished triangle in $Sm/k$. 
Moreover, since $\sY(V\times \A^1) \to \sY(V)$ is a weak equivalence for every $V/U$. Thus by arguments given as in \cite[A.6]{morel}, $\sY_{|U}$ is $\A^1$-local. 
\end{proof}

\begin{proof}[Proof of \ref{main theorem}]
Let $K$ denote the kernel of the the epimorphism $G\to H$. Thus we have a short exact sequence of Nisnevich sheaves of groups 
$$ 1 \to K \to G \to H \to 1. $$

\noindent \underline{Step 1}:
For every smooth $k$-scheme $U$, this gives us an exact sequence (see \cite[3.3.1]{giraud})  of pointed cohomology sets
$$ 1 \to H^0(U,K) \to H^0(U, G) \to H^0(U, H) \to H^1(U, K) \to H^1(U, G) \to H^1(U, H)$$
Using functoriality of the above exact sequence in the case when $U$ is Hensel local, we deduce that the $\A^1$-invariance of $H$ implies (in fact is equivalent to) strong $\A^1$-invariance of $K$. By Theorem \ref{center}, $Z(K)$, the center of $K$ is strictly $\A^1$-invariant sheaf. \\

\noindent \underline{Step 2}:  To show strong $\A^1$-invariance of $H$, it is enough to show that for all henselian local essentially smooth schemes $U/k$, $H^1(U\times \A^1_k,H)$ is trivial. By \cite[4.7.2.4]{giraud}, we have an exact sequence of pointed sets
$$ \to H^1(U\times \A^1_k, G) \to H^1(U\times \A^1_k, H) \to O(\phi) $$
where $\phi$ denotes restriction of $G\to H$ to the over-category $(Sm/k)/U\times\A^1$ which will be denoted by $Sm_k/U\times \A^1$ for simplicity and $O(\phi)$ is as defined in \ref{o}. It is enough to show $O(\phi)$ is trivial. This follows from 
Lemma \ref{o-trivial} and the strict $\A^1$-invariance of $Z(K)$.

\end{proof}

Let $\sF$ be a sheaf on $Sm/k$. For $U\in Sm/k$ we say $\alpha, \beta \in \sF(U)$ are naive $\A^1$-homotopic if there exists a $\gamma \in \sF(U\times \A^1)$ such that 
$$ \alpha = \sigma_0^*(\gamma) \ \ \text{and} \ \ \beta = \sigma_1^*(\gamma)$$
where $\sigma_0, \sigma_1 : U\to U\times \A^1$ are the sections defined by $0$ and $1$ respectively. As in  \cite[2.9]{CAA}, let $S(\sF)$ denote the sheaf associated to the presheaf 
$$ U \mapsto \frac{\sF(U)}{\sim} $$
where $\sim$ denotes the equivalence relation generated by naive $\A^1$-homotopies. There is a canonical epimorphism 
$ \sF \to S(\sF)$. Let 
$$S^{\infty}(\sF) := \lim_{n \to \infty} S^n(\sF)$$
The following lemma is straightforward to check.
\begin{lemma} \label{caa} (see \cite[2.13]{CAA}) The canonical morphism $\sF \to S^{\infty}(\sF)$ is an epimorphism and is a universal map from $\sF$ to an $\A^1$-invariant sheaf. Moreover if $\sF$ is a sheaf of groups, then so is $S^{\infty}(\sF)$. 
\end{lemma}

\begin{proof}[Proof of Theorem \ref{hurewicz}]
By lemma \ref{caa}, the map $$ \pi_1^{\A^1}(\sX)\xrightarrow{h} H_1^{\A^1}(\sX)$$
 factors uniquely through 
$$\pi_1^{\A^1}(\sX) \xrightarrow{s} S^{\infty}\left( \pi_1^{\A^1}(\sX)^{ab} \right).$$
Since the map $s$ is an epimorphism, Theorem \ref{main theorem} implies that  $S^{\infty}\left( \pi_1^{\A^1}(\sX)^{ab} \right)$ is strongly $\A^1$-invariant sheaf of abelian groups. Thus $s$ is a universal map to strictly $\A^1$-invariant sheaf of abelian groups. However by \cite[6.35]{morel}, so is $h$. Thus the induced map from $S^{\infty}\left( \pi_1^{\A^1}(\sX)^{ab} \right) \to H_1^{\A^1}(\sX)$ must be an isomorphism. In particular $h$ must be an epimorphism. 
\end{proof}

%%%%%%%
\section{Hurewicz map for $\P^1_k$}
In this section we reserve the notation $H$ to denote the Hurewicz map for $\P^1$, i.e. $ \pi_1^{\A^1}(\P^1) \xrightarrow{H} H^{\A^1}_1(\P^1)$. The goal of this section is to prove the following propositions

\begin{proposition}\label{hurewicz-p1}
The kernel of the Hurewicz map for $\P^1_k$, 
$$ \pi_1^{\A^1}(\P^1) \xrightarrow{H} H^{\A^1}_1(\P^1)$$ 
 is equal to the commutator subgroup of $\pi_1^{\A^1}(\P^1)$. 
 \end{proposition}

As a consequence of the explicit computation of the Hurewicz map we obtain the following : 
\begin{proposition}\label{etah}
The sequence of Nisnevich sheaves $$ 0 \to h\underline{K}_2^{MW} \to \underline{K}_2^{MW} \xrightarrow{\eta} \underline{K}_1^{MW}$$ is exact.
\end{proposition}

\begin{remark}
We do not know if there is an elementary way to prove the above proposition, using generator and relations. In particular we do not know if 
$$ 0 \to h\underline{K}_n^{MW} \to \underline{K}_n^{MW} \xrightarrow{\eta} \underline{K}_{n-1}^{MW}$$
is exact for every $n\geq 1$. However, as pointed out to us by O. R{\"o}ndigs, it is possible that the above short exact sequence is induced by a cofiber sequence given in \cite[Prop. 11]{roendigs}. 
\end{remark}

The most difficult part of the computation in the above propositions is the universality of the Hurewicz map and the computation of $\pi_1^{\A^1}(\P^1)$ itself, both of which has been elegantly done in \cite[6.35, 7.3]{morel}. We first restate Morel's computation of $\pi_1^{\A^1}(\P^1)$ as it will also help us to build notation for use in subsequent calculation. 

Let $F^{\A^1}(1) := \pi_1^{\A^1}(\P^1)$. First we recall the following two maps defined by Morel:
\begin{enumerate}
\item A map $\theta: \G_m \to F^{\A^1}(1)$ which is a result of an $\A^1$-equivalence $\P^1_k \sim \Sigma(\G_m)$. 
\item A map $\underline{K}_2^{MW} \to F^{\A^1}(1)$ which is a result of applying $\pi_1$ to the map $\A^2-0 \to \P^1$ and a theorem of Morel which shows $\underline{K}_2^{MW} \cong \pi_1^{\A^1}(\A^2-0)$.
\end{enumerate}

In what follows, we will freely use standard notation for denoting elements of $K_*^{MW}$ used in \cite[Chapter 3]{morel}, e.g.  $\left< -1\right>, h, [U]$ etc. 

\begin{theorem}\cite[7.29]{morel} \label{fa1} As a sheaf of sets $F^{\A^1}(1)$ is a product $K_2^{MW}\times \G_m$ and the following describes the structure of $F^{\A^1}(1)$ completely:
\begin{enumerate}
\item[(i)] The sequence $$ 1\to \underline{K}_2^{MW} \to F^{\A^1}(1) \xrightarrow{\gamma} \G_m \to 1$$ is exact and is a central extension.
\item[(ii)] For two units $U,V$ in any field extension $F$ of $k$, the following hold
$$ \theta(U)\theta(V)^{-1} = [-U][-V]\theta(U^{-1}V)^{-1}$$
$$ \theta(U)^{-1}\theta(V) = [U^{-1}][-V] \theta(U^{-1}V).$$
\end{enumerate}
\end{theorem}

The following is the main calculation in the proof of Proposition \ref{hurewicz-p1}. 
\begin{lemma}\label{hk2} For any essentially smooth field extension $F/k$, the commutator subgroup of $F^{\A^1}(1)(F)$ is equal to  
$h K_2^{MW}(F)$. 
\end{lemma}
\begin{proof}
Since $K_2^{MW}(F)$ is in the center of $F^{\A^1}(1)(F)$, we have that the commutator subgroup 
$$[F^{\A^1}(1)(F), F^{\A^1}(1)(F)] = \left< \theta(U)\theta(V)\theta(U)^{-1}\theta(V)^{-1} | \ U, V \in F\right>$$
where the angle brackets on RHS denote subgroup generated by the elements within. 
$${
\begin{split}
 \theta(U)\theta(V) & = \left<-1\right>[U][V]\theta(UV)  & ...(\text{by \cite[7.31]{morel}}) \\
 \theta(V)\theta(U) & =  \left<-1\right>[U][V]\theta(UV) & \\
 \theta(U)\theta(V)\theta(U)^{-1}\theta(V)^{-1}  & = \left<-1\right>[U][V]\theta(UV) \cdot \big(-\left<-1\right>[V][U]\theta(UV)^{-1}\big) & ...(\text{by \cite[7.29(ii)]{morel}})  \\
     & = \left<-1\right>\big([U][V] -  [V][U]) & ...(\because K_2^{MW} \text{is in the center})\\ 
     & = [U][V]\big(\left<-1\right> + \left<-1\right>^2\big) & ...(\text{using \cite[3.7(3)]{morel}}) \\
     & = h(h-1)[U][V] & ...(\because h=1+\left<-1\right>)
\end{split}
}$$
Thus we have 
$${
\begin{split}
[F^{\A^1}(1)(F), F^{\A^1}(1)(F)] & = \left< h(h-1)[U][V] \ | \ U, V \in F\right> & \\
	& = h(h-1)K_2^{MW}(F) & ...( \text{using \cite[3.6(1)]{morel}})\\ 
	& = h\cdot K_2^{MW}(F) & ...(\because h-1=\left<-1\right> \text {is a unit by \cite[3.5(4)]{morel}})
\end{split}
}$$
\end{proof}

\begin{proof}[Proof of Proposition \ref{hurewicz-p1}]
Recall that $H$ is a universal map from $\pi_1^{\A^1}(\P^1)$ to a strongly $\A^1$-invariant sheaf of abelian groups. Thus  
it is enough to show that the abelianisation of $\pi_1^{\A^1}(\P^1)$ is strongly $\A^1$-invariant. $\pi_1^{\A^1}(\P^1)^{ab}$  is a quotient of a strongly $\A^1$-invariant sheaf therefore by Theorem \ref{main theorem}, we need to show that  $\pi_1^{\A^1}(\P^1)^{ab}$ is homotopy invariant.   However by \eqref{fa1} and Lemma \ref{hk2}, $\pi_1^{\A^1}(\P^1)^{\rm ab}$ fits in the following exact sequence 

$$ 0 \to \frac{\underline{K}_2^{MW}}{h\underline{K}_2^{MW}} \to \pi_1^{\A^1}(\P^1)^{ab} \to \G_m \to 0.$$
Thus it is enough to show that $\frac{\underline{K}_2^{MW}}{h\underline{K}_2^{MW}}$ is $\A^1$-invariant or equivalently $h\underline{K}_2^{MW}$ is strongly $\A^1$-invariant.
However this follows from \cite[3.32]{morel}.
\end{proof}

The following lemmas give explicit description of the Hurewicz morphism $H$. 

\begin{lemma}\label{explicit-h} There exists an isomorphism $\phi : H_1^{\A^1}(\P^1)\cong \underline{K}_1^{MW}$ such that 
$$ \phi \circ H([U][V]\theta(W)) = \eta[U^{-1}][V] + [W] $$
where $U,V,W$ are sections of $\G_m$ on an object in $\G_m$. 
\end{lemma}
\begin{proof}
Since there is an $\A^1$-equivalence $\P^1_k \cong \Sigma \G_m $, we have 
$$ H_1^{\A^1}(\P^1_k) \cong H_1^{\A^1}(\Sigma \G_m) \cong \tilde{H}_0^{\A^1}(\G_m)$$
where the second isomorphism is due to $\A^1$-suspension theorem for homology \cite[6.30]{morel}. In the above statement  $\G_m$ is considered as a sheaf of sets pointed by $1$. 
By the definition of $\A^1$-homology groups, $\tilde{H}_0^{\A^1}(\G_m)$ is the strictly $\A^1$-invariant sheaf of abelian groups generated by the pointed sheaf $\G_m$. By \cite[3.2]{morel}, this must be isomorphic to $\underline{K}_1^{MW}$. Now recall that we have the following commutative diagram 
$$\xymatrix{
\G_m \ar[r] \ar[d]_{U\mapsto [U]}\ar@/^2pc/[rr]^\theta & \pi_1(\Sigma \G_m) \ar[r]  & \pi_1^{\A^1}(\P^1_k)\ar[d]^H \\
	\underline{K}_1^{MW}	 & \tilde{H}_0^{\A^1}(\G_m)\ar[l]^-{\sim} & H_1^{\A^1}(\Sigma \G_m) \ar[l]^-{\sim} \ar@/^2pc/[ll]^\phi
}$$
The diagram commutes because all morphisms in this diagram are a result of some universal property.  The commutativity of the diagram gives us the formula 
\begin{equation}\phi(H([W])) = [W]
\end{equation}
Now, using the following equality proved in \cite[7.29(1)]{morel}
 $$ \theta(U^{-1})^{-1} \theta(U^{-1}V) \theta(V)^{-1} = [U][V] $$
 we get
 $${
 \begin{split}
 \phi(H([U][V]\theta(W))) & = \phi( H (\theta(U^{-1})^{-1} \theta(U^{-1}V) \theta(V)^{-1}\theta(W))) &  \\
                                       & = -[U^{-1}] + [U^{-1}V] -[V] + [W] & ...(\text{using eqn(1) above}) \\
                                       & = \eta[U^{-1}][V] + [W] & ...(\text{using \cite[3.1(2)]{morel}})
 \end{split}
 }$$
\end{proof}

To state the next lemma, recall that $\underline{K}_2^{MW}$ is the free strongly $\A^1$-invariant sheaf of abelian groups generated by the pointed sheaf of sets $\G_m\wedge \G_m$. Thus any automorphism of $\G_m\wedge \G_m$ gives rise to an automorphism of $\underline{K}_2^{MW}$. In particular, the automorphism of $\underline{K}_2^{MW}$ induced by 
$${
\begin{split}
	\G_m\wedge \G_m & \longrightarrow  \G_m\wedge \G_m \\
	(U, V) & \mapsto (U^{-1}, V)
\end{split}
}$$ will be denoted by $\tau$. 
\begin{lemma} \label{hurewicz-decoded}
Let $\tau: \underline{K}_2^{MW} \to \underline{K}_2^{MW}$ denote the automorphism which sends $[U][V] \mapsto [U^{-1}][V]$. Then the restriction of the Hurewicz map $H$ to $\underline{K}_2^{MW}$ coincides with the composition 
$$ \underline{K}_2^{MW} \xrightarrow{\tau} \underline{K}_2^{MW} \xrightarrow{\cdot \eta} \underline{K}_1^{MW}$$
\end{lemma}
\begin{proof}
This follows directly from the explicit formula for $H$ in the above lemma. 
\end{proof}

\begin{proof}[Proof of Proposition \ref{etah}]
By universality of $H$, as noted before, the map $F^{\A^1}(1)\xrightarrow{\gamma} \G_m$ factors through $H$. Thus 
$$\Ker(H) \subset \Ker(\gamma) = \underline{K}_2^{MW}.$$ 
$\Ker(H)$ must coincide with the kernel of restriction of $H$ to $\underline{K}_2^{MW}$. Now the proposition follows from Lemmas \ref{hk2} and \ref{hurewicz-decoded}.
\end{proof}

\end{document}